\documentclass[a4paper,10pt]{article}
\usepackage[dvipdfmx]{graphicx}
\usepackage{float}
\usepackage{latexsym,amssymb,amsmath,mathrsfs}
\usepackage{amsthm}
\usepackage{mathrsfs}
\usepackage[english]{babel}
\usepackage{comment}
\newtheorem{theorem}{Theorem}
\newtheorem*{theorem*}{Theorem}
\newtheorem{lemma}{Lemma}
\newtheorem*{lemma*}{Lemma}
\theoremstyle{definition}
\newtheorem{remark}{Remark}
\newtheorem*{remark*}{Remark}
\newtheorem{corollary}{Corollary}
\newtheorem*{corollary*}{Corollary}
\DeclareMathOperator{\Ric}{Ric}

\newcommand{\R}{\mathbb{R}}

\newcommand{\td}{\mathrm{d}}
 
\newcommand{\lra}{\longrightarrow}

\newcommand{\ez}{\varepsilon}

\newcommand{\bs}{\ensuremath{\mathbf{s}}}
\newcommand{\e}{\mathrm{e}}

\renewcommand{\d}{\mathrm{d}}

\newcommand{\eps}{{\varepsilon}}

\title{Some functional inequalities under lower Bakry-\'{E}mery-Ricci curvature bounds with $\eps$-range}
\author{Yasuaki Fujitani\thanks{Department of Mathematics, Osaka University, Osaka 560-0043, Japan (\texttt{u197830k@ecs.osaka-u.ac.jp})}}
\begin{document}
\maketitle
\begin{abstract}
For $n$-dimensional weighted Riemannian manifolds, lower $m$-Bakry-\'{E}mery-Ricci curvature bounds with $\eps$-range, introduced by Lu-Minguzzi-Ohta \cite{lu}, integrate constant lower bounds and certain variable lower bounds in terms of weight functions.
%The condition of lower bounds of $m$-Bakry-\'{E}mery-Ricci curvature with $\eps$-range integrates the conditions of constant lower bounds and variable lower bounds.
In this paper, we prove a Cheng type inequality and a local Sobolev inequality under lower $m$-Bakry-\'{E}mery-Ricci curvature bounds with $\eps$-range. % in the case of the effective dimension $m\in(-\infty,1]\cup [n,\infty]$.
These generalize those inequalities under constant curvature bounds for $m \in (n,\infty)$ to $m\in(-\infty,1]\cup\{\infty\}$.
\end{abstract}
\tableofcontents
\section{Introduction}
The Ricci curvature plays an important role in geometric analysis.
For example, lower bounds of Ricci curvature imply comparison theorems such as the Laplacian comparison theorem and Bishop-Gromov volume comparison theorem. This paper is concerned with the Bakry-\'{E}mery-Ricci curvature $\Ric_\psi^m$, which is a generalization of the Ricci curvature for weighted Riemannian manifolds and $m$ is a real parameter called the effective dimension.
The condition $\Ric_\psi^m\geq K$ for $K\in\mathbb{R}$ implies many comparison geometric results similar to those for Riemannian manifolds with Ricci curvature bounded from below by $K$ and dimension bounded from above by $m$.
Especially the case of $m\geq n$ is now classical and well investigated.
Recently, there is a growing interest in the $m$-Bakry-\'{E}mery-Ricci curvature in the case of $m \in (-\infty,1]$.
For this range, some Poincar\'{e} inequalities \cite{milman} (see also \cite{mai} for its rigidity), Beckner inequality \cite{gentil} and the curvature-dimension condition \cite{ohta-negative} were studied. 
%In \cite{mai}, it turned out that positive lower bounds of the Bakry-\'{E}mery-Ricci curvature in some negative effective dimension have interesting geometric rigidity results different from the case when the effective dimension is positive.
%It is known that constant lower bounds of the $m$-Bakry-\'{E}mery-Ricci curvature do not imply comparison results such as the Laplacian comparison theorem and Bishop-Gromov volume comparison theorem in the case of the effective dimension $m\in(-\infty,1]$.
 
%Still, some comparison geometry is known when the lower bounds of the $m$-Bakry-\'{E}mery-Ricci curvature are not constants even in the case of $m \in (-\infty,1]$.
%The case of $m = 1$ is investigated in \cite{yoroshikin} and they proved some comparison theorems. After that, \cite{kuwae_li} generalized results in \cite{yoroshikin} in the case of $m\in (-\infty,1]$.
%In \cite{lu}, it turned out that lower bounds of the $m$-Bakry-\'{E}mery-Ricci curvature with $\eps$-range imply some comparison theorems in the case of $m\in(-\infty,1]\cup [n,\infty]$, which is not only a generalization of \cite{yoroshikin} and \cite{kuwae_li}, but also a unification including the constant curvature bounds.
It is known that some comparison theorems (such as the Bishop-Gromov volume comparison theorem and the Laplacian comparison theorem) under the constant curvature bound $\Ric_\psi^m\geq Kg$ hold only for $m\in[n,\infty)$ and fail for $m\in (-\infty,1]\cup \{\infty\}$.
Nonetheless, Wylie-Yeroshkin \cite{yoroshikin} introduced a variable curvature bound 
\begin{equation*}
    \Ric_\psi^1 \geq K\e^{-\frac{4}{n-1}\psi}g
\end{equation*}
associated with the weight function $\psi$, and established several comparison theorems. They were then generalized to 
\begin{equation*}
    \Ric_\psi^m\geq K\e^{\frac{4}{m-n}\psi}g
\end{equation*}
with $m\in (-\infty, 1)$ by Kuwae-Li \cite{kuwae_li}. In \cite{lu}, Lu-Minguzzi-Ohta gave a further generalization of the form 
\begin{equation*}
    \Ric_\psi^m\geq K\e^{\frac{4(\eps-1)}{n-1}\psi}g
\end{equation*}
for an additional parameter $\eps$ in an appropriate range, depending on $m$, called the $\eps$-range (see also \cite{lu2} for a proceding work on singularity theorems in Lorentz-Finsler geometry).
This is not only a generalization of \cite{yoroshikin} and \cite{kuwae_li}, but also a unification of both constant and variable curvature bounds by choosing appropriate $\eps$.
We refer to \cite{kuwae_sakurai,kuwae_sakurai_2,kuwae_sakurai_3} for further investigations on the $\eps$-range.

In this paper, we assume lower bounds of the $m$-Bakry-\'{E}mery-Ricci curvature with $\eps$-range and study analytic applications on non-compact manifolds. The main contributions of this paper are the following:
\begin{itemize}
\item We give an upper bound of the $L^p_\mu$-spectrum. In particular, when $p = 2$, this gives an upper bound of the first nonzero eigenvalue of the weighted Laplacian.
\item We give an explicit form of a local Sobolev inequality. %(See Remark \ref{last} for further comments.)
\end{itemize}

An upper bound of the first nonzero eigenvalue of the Laplacian under lower Ricci curvature bounds was first investigated in \cite{cheng} in 1975 and it is called the Cheng type inequality.
Some variants of the Cheng type inequality are known (we refer to \cite{linfengwang1}, for example) under lower bounds of the $m$-Bakry-\'{E}mery-Ricci curvature in the case of $m\in[n,\infty]$.
Our Theorem \ref{cheng-type-epsilon-range} generalizes them.
The local Sobolev inequality is an important tool for the DeGiorgi-Nash-Moser theory.
Recently in \cite{linfengwang2}, they obtained a Liouville type theorem for the weighted $p$-Laplacian by using a local Sobolev inequality and Moser's iteration techniques.
Our results in Theorem \ref{sobolev_thm} are consistent with the local Sobolev inequality in \cite{linfengwang2} in the case of constant curvature bounds and the effective dimension $m\in[n,\infty)$.

This paper is organized as follows. In Section 2, we briefly review the Bakry-\'{E}mery-Ricci curvature and Cheng type inequalities and local Sobolev inequalities. We show a Cheng type inequality in Section 3 and a local Sobolev inequality in Section 4 under lower bounds of the Bakry-\'{E}mery-Ricci curvature with $\eps$-range.
In Appendix, we give a variant of Cheng type inequality for deformed metrics under lower bounds of the Bakry-\'{E}mery-Ricci curvature with $\eps$-range.
\section{Preliminaries}
\subsection{$\eps$-range}\label{epsilon-range-section}
Let $(M,g,\mu)$ be an $n$-dimensional weighted Riemannian manifold.
We assume that $M$ is non-compact in this paper.
We set $\mu = \e^{-\psi}v_g$ where $v_g$ is the Riemannian volume measure and $\psi$ is a $C^{\infty}$ function on $M$.
For $m\in(-\infty,1]\cup[n,+\infty]$, the \emph{$m$-Bakry-\'{E}mery-Ricci curvature} is defined as follows:
\begin{align*}
\Ric_\psi^m:=\Ric_g+{\rm \nabla^2}\psi-\frac{\d\psi\otimes \d\psi}{m-n},
\end{align*}
where when $m=+\infty$,
the last term is interpreted as the limit $0$
and when $m=n$,
we only consider a constant function $\psi$, and
set $\Ric_\psi^n:=\Ric_g$.
 
In \cite{lu}, \cite{lu2}, they introduced the notion of \textit{$\eps$-range}:
\begin{equation}\label{epsilin-range}
\eps = 0 \mbox{ for } m = 1, \quad |\eps| < \sqrt{\frac{m-1}{m-n}} \mbox{ for } m\neq 1,n, \quad \eps\in \mathbb{R}\mbox{ for } m = n.
\end{equation}
In this $\eps$-range, for $K\in\mathbb{R}$, they considered the condition
\[ \Ric_\psi^m(v)\ge K \e^{\frac{4(\ez-1)}{n-1}\psi(x)} g(v,v),\quad v\in T_xM.\]
We also define the associated constant $c$ as
\begin{equation}\label{c_num}
c = \frac{1}{n-1}\left(1- \eps^2\frac{m-n}{m-1}\right) > 0
\end{equation}
for $m\neq 1$ and $c = (n-1)^{-1}$ for $m = 1$.
We define the comparison function $\bs_{\kappa}$ as
\begin{equation}\label{eq:bs}
\bs_{\kappa}(t) := \begin{cases}
\frac{1}{\sqrt{\kappa}} \sin(\sqrt{\kappa}t) & \kappa>0, \\
t & \kappa=0, \\
\frac{1}{\sqrt{-\kappa}} \sinh(\sqrt{-\kappa}t) & \kappa<0.
\end{cases}
\end{equation}
We denote $B(x,r)=\{ y \in M \,|\, d(x,y)<r \}$, $V(x,r) = \mu(B(x,r))$ and $tB = B(x,tr)$ if $B = B(x,r)$.
\begin{theorem}(\cite[Theorem 3.11]{lu}, Bishop-Gromov volume comparison theorem)\label{bishop_gromov}
Let $(M,g,\mu)$ be a complete weighted Riemannian manifold and $m \in (-\infty,1] \cup [n,+\infty]$,
$\ez \in \R$ in the $\ez$-range \eqref{epsilin-range}, $K \in \R$ and $b \ge a>0$.
Assume that
\[ \Ric_\psi^m(v)\ge K \e^{\frac{4(\ez-1)}{n-1}\psi(x)} g(v,v)\]
holds for all $v \in T_xM \setminus 0$ and
\[ a \le \e^{-\frac{2(\ez-1)}{n-1}\psi} \le b. \]
Then we have
\[ \frac{\mu(B(x,R))}{\mu(B(x,r))}
\le \frac{b}{a}
\frac{\int_0^{\min\{R/a,\,\pi/\sqrt{cK}\}} \bs_{cK}(\tau)^{1/c} \,\d\tau}{\int_0^{r/b} \bs_{cK}(\tau)^{1/c} \,\d\tau} \]
for all $x\in M$ and $0<r<R$, where $R \le b\pi/\sqrt{cK}$ when $K>0$
and we set $\pi/\sqrt{cK}:=\infty$ for $K \le 0$.
\end{theorem}
We briefly review the argument in \cite{lu} (where they considered, more generally, Finsler manifolds equipped with measures).
Given a unit tangent vector $v \in T_xM $,
let $\eta:[0,l) \lra \R$ be the geodesic with $\dot{\eta}(0)=v$.
We take an orthonormal basis $\{e_i\}_{i=1}^n$ of $T_xM$ with $e_n=v$
and consider the Jacobi fields
\[ E_i(t):=(\td\exp_x)_{tv}(te_i), \quad i=1,2,\ldots,n-1, \]
along $\eta$.
Define the $(n-1) \times (n-1)$ matrices $A(t)=(a_{ij}(t))$ by
\[ a_{ij}(t):=g \big( E_i(t),E_j(t) \big). \]
We define
\[ h_0(t):=(\det A(t))^{1/2(n-1)}, \qquad
h(t) := \e^{-c\psi(\eta(t))}\big(\! \det A(t) \big)^{c/2}, \qquad
h_1(\tau):=h \big( \varphi_\eta^{-1}(\tau) \big) \]
for $t \in [0,l)$ and $\tau \in [0,\varphi_{\eta}(l))$, where
\begin{equation*}
\varphi_{\eta}(t) :=\int_0^t \e^{\frac{2(\ez -1)}{n-1}\psi(\eta(s))} \,\td s.
\end{equation*}
By the definition, we have the following relationship:
\[ (\e^{-\psi(\eta)} h_0^{n-1})(t) =h(t)^{1/c} =h_1\big( \varphi_{\eta}(t) \big)^{1/c}. \]
According to the argument in \cite[Theorem 3.6]{lu}, the condition $\Ric_\psi^m(v)\ge K \e^{\frac{4(\ez-1)}{n-1}\psi(x)} g(v,v)$ implies that
\begin{equation}\label{non_increasing_property}
(\e^{-\psi(\eta)} h_0^{n-1})/\bs_{cK}(\varphi_{\eta})^{1/c} \mbox{ is non-increasing.}
\end{equation}
This plays the key role in proving Theorem \ref{bishop_gromov} above.

\subsection{Upper bounds of the $L^p_\mu$-spectrum}
In this subsection, we explain Cheng type inequalities under lower Bakry-\'{E}mery-Ricci curvature bounds by constants. We generalize these results to the $\eps$-range in Section 3.
For $p > 1$, the \emph{$L^p_\mu$-spectrum} is defined by
\begin{equation*}
\lambda_{\mu,p}(M) := \inf_{\phi\in C_0^{\infty}(M)}\frac{\int_M |\nabla \phi|^p \d\mu}{\int_M|\phi|^p \d\mu}.
\end{equation*}
When $p = 2$, the $L^p_\mu$ spectrum is the first nonzero eigenvalue of the weighted Laplacian.
Under lower $m$-Bakry-\'{E}mery-Ricci curvature bounds with $m \in [n,\infty)$, we have the following theorems.
\begin{theorem}(\cite[Theorem 3.2]{linfengwang1})\label{cheng-type-m-inequality}
Let $(M, g, \mu)$ be an $n$-dimensional weighted complete Riemannian manifold.
Assume that $\Ric_\psi^m \geq- K\ (K \geq 0)$. Then the $L^p_\mu$-spectrum satisfies
$$
\lambda_{\mu,p}(M) \leq \left(\frac{\sqrt{(m-1)K}}{p}\right)^p.
$$
\end{theorem}
An additional assumption on the weight function leads to the following Cheng type inequality under a lower $\infty$-Bakry-\'{E}mery-Ricci curvature bound.
\begin{theorem}(\cite[Theorem 3.3]{linfengwang1})
Let $(M, g, \mu)$ be an $n$-dimensional complete weighted Riemannian manifold.
We fix a point $q\in M$.
Assume that
$\Ric_\psi^{\infty} \geq- K \ (K \geq 0)$
and
$ \frac{\partial \psi}{\partial r} \geq-k\ (k \geq 0)$
along all minimal geodesic segments from the fixed point $q\in M$, where $r$ is the distance from $q$. Then the $L^p_\mu$-spectrum satisfies
$$
\lambda_{\mu,p}(M) \leq \left(\frac{ \sqrt{(n-1)K}+k}{p}\right)^p.
$$
\end{theorem}
These results are generalizations of the original Cheng type inequality in \cite{cheng}.
\subsection{Local Sobolev inequality}
%Local Sobolev inequality is studied in \cite{soliton}, \cite{wu3} under lower bounds of the $\infty$-Bakry-\'{E}mery Ricci curvature with some additional conditions on the weight function.
We have the following local Sobolev inequality under lower bounds of the $m$-Bakry-\'{E}mery-Ricci curvature in the case of $m\in (n,\infty)$ and $n \geq 2$. %under lower bounds of the $m$-Bakry-\'{E}mery-Ricci curvature when $m > n \geq 2$.
We generalize the following result in Section 4. We refer to \cite{soliton} for the case of $m = \infty$.
\begin{theorem}(\cite[Lemma 3.2]{linfengwang2})
Let $\left(M, g, \mu\right)$ be an $n$-dimensional weighted complete Riemannian manifold. If $\operatorname{Ric}_\psi^m \geq-(m-1) K$ for some $K\geq 0$ and $m>n \geq 2$, then there exists a constant $C$, depending on $m$, such that for all $B(o,r) \subset M$ we have for $f \in C_0^{\infty}\left(B(o,r)\right)$,
$$
\left(\int_{B(o,r)}|f|^{\frac{2 m}{m-2}} \mathrm{~d} \mu\right)^{\frac{m-2}{m}} \leq \e^{C(1+\sqrt{K} r)}\mu(B(o, r))^{-\frac{2}{m}} r^2 \int_{B(o,r)}\left(|\nabla f|^2+r^{-2} f^2\right) \mathrm{d} \mu.
$$
\end{theorem}
We will use the next theorem in Subsection 4.2 to prove a local Sobolev inequality under lower Bakry-\'{E}mery-Ricci curvature bounds with $\eps$-range.
%with $A = -\Delta_\psi$
%%\begin{theorem}(\cite[Theorem 2.2]{salof1})\label{proof_ingredients_analysis}
% Let $\e^{-tA}$ be a symmetric submarkovian semigroup acting on the $L^p(M,\mu)$. Given $\nu > 2$, the following properties are equivalent.
% \begin{itemize}
% \item $\|f\|^2_{\frac{2\nu}{\nu-2}}\leq C_1(\|A^{\frac{1}{2}f}\|_2^2 + t_0^{-1}\|f\|_2^2) ,$
% \item $\|f\|_2^{2 + \frac{4}{\nu}}\leq C_2(\|A^{\frac{1}{2}}f\|_2^2 + t_0^{-1}\|f\|_2^2)\|f\|_1^{\frac{4}{\nu}},$
% \end{itemize}
% where $C_1 = \nu C_0 C_2$ for some numerical constant $C_0>0$.
%\end{theorem}
 
\begin{theorem}(\cite[Theorem 2.2]{salof1})\label{proof_ingredients_analysis}
Let $\e^{-t A}$ be a symmetric submarkovian semigroup acting on the spaces $L^p(M, \d\mu)$. Given $\nu>2$, the following three properties are equivalent.
\begin{itemize}
\item[\mbox{1}.] $\left\|\e^{-t A} f\right\|_{\infty} \leq C_0 t^{-\nu / 2}\|f\|_1$ for $0<t<t_0$.
\item[\mbox{2}.] $\|f\|_{2 \nu /(\nu-2)}^2 \leq C_1\left(\left\|A^{1 / 2} f\right\|_2^2+t_0^{-1}\|f\|_2^2\right)$.
\item[\mbox{3}.] $\|f\|_2^{2+4 / \nu} \leq C_2\left(\left\|A^{1 / 2} f\right\|_2^2+t_0^{-1}\|f\|_2^2\right)\|f\|_1^{4 / \nu}$.
\end{itemize}
Moreover, 3. implies 1. with $C_0=\left(\nu C C_2\right)^{\nu / 2}$ and 1. implies 2. with $C_1=C C_0^{2 / \nu}$, where $C$ is some numerical constant.
\end{theorem}

\section{Upper bound of the $L^p_\mu$-spectrum with $\eps$-range}
\begin{theorem}\label{cheng-type-epsilon-range}
Let $(M,g,\mu)$ be an $n$-dimensional weighted complete Riemannian manifold and $m \in (-\infty,1] \cup [n,+\infty]$,
$\ez \in \R$ in the $\ez$-range \eqref{epsilin-range}, $K > 0$ and $b \ge a>0$.
Assume that
\[ \Ric_\psi^m(v)\ge -K \e^{\frac{4(\ez-1)}{n-1}\psi(x)} g(v,v) \]
holds for all $v \in T_xM \setminus 0$ and
\begin{equation}\label{eq:wab}
a \le \e^{-\frac{2(\ez-1)}{n-1}\psi} \le b.
\end{equation}
Then, for $p > 1$, we have
\begin{equation*}
\lambda_{\mu,p}(M) \leq \left(\sqrt{\frac{K}{c}}\frac{1}{pa}\right)^p.
\end{equation*}
\end{theorem}
\begin{proof}
We apply the argument in \cite[Theorem 3.2]{linfengwang1}. For an arbitrary $\delta > 0$, we set
\begin{equation*}
\alpha := -\frac{\sqrt{\frac{K}{c}}\frac{1}{a} + \delta}{p}
\end{equation*}
and, for $x\in M$ and $R\geq 2$,
\begin{equation*}
\phi(y) := \exp(\alpha r(y))\varphi(y),
\end{equation*}
where $r(y) = d(x,y)$ and $\varphi$ is a cut off function on $B(x,R)$ such that $\varphi = 1$ on $B(x,R-1)$, $\varphi = 0$ on $M\setminus B(x,R)$ and $|\nabla \varphi| \leq C_3$, where $C_3$ is a constant independent of $R$. For an arbitrary $\zeta > 0$, we have
\begin{eqnarray*}
|\nabla \phi|^p &=&\left|\alpha \e^{\alpha r} \varphi \nabla r+\e^{\alpha r} \nabla \varphi\right|^p \\
& \leq& \e^{p \alpha r}(-\alpha \varphi+|\nabla \varphi|)^p \\
& \leq& \e^{p \alpha r}\left[(1+\zeta)^{p-1}(-\alpha \varphi)^p+\left(\frac{1+\zeta}{\zeta}\right)^{p-1}|\nabla \varphi|^p\right] .
\end{eqnarray*}
By the definition of $\lambda_{\mu,p}(M)$, we find
\begin{eqnarray}
\lambda_{\mu,p}(M) & \leq&(1+\zeta)^{p-1}(-\alpha)^p+\left(\frac{1+\zeta}{\zeta}\right)^{p-1} \frac{\int_M \e^{p \alpha r}|\nabla \varphi|^p \d\mu}{\int_M \e^{p \alpha r} \varphi^p \d\mu}\nonumber \\
&=&(1+\zeta)^{p-1}(-\alpha)^p+\left(\frac{1+\zeta}{\zeta}\right)^{p-1} \frac{\int_{B(x,R) \backslash B(x,R-1)} \e^{p \alpha r}|\nabla \varphi|^p \d\mu}{\int_{B(x,R)} \e^{p \alpha r} \varphi^p \d\mu}\nonumber \\
& \leq&(1+\zeta)^{p-1}(-\alpha)^p+C_3^p\left(\frac{1+\zeta}{\zeta}\right)^{p-1} \frac{\e^{p \alpha (R-1)} \mu\left(B(x,R)\right)}{\int_{B(x,1)} \e^{p \alpha r} \d\mu} \nonumber\\
& \leq&(1+\zeta)^{p-1}(-\alpha)^p+C_3^p\left(\frac{1+\zeta}{\zeta}\right)^{p-1} \frac{\e^{p \alpha (R-1)} \mu\left(B(x,R)\right)}{\e^{p \alpha} \mu\left(B(x,1)\right)}\label{moto-15} .
\end{eqnarray}
It follows from Theorem \ref{bishop_gromov} that
\begin{equation}\label{volume_comparison_1}
\mu(B(x,R))\leq \mu(B(x,1))\frac{b}{a}
\frac{\int_0^{R/a} \bs_{-cK}(\tau)^{1/c} \,\d\tau}{\int_0^{1/b} \bs_{-cK}(\tau)^{1/c} \,\d\tau} .
\end{equation}
To estimate the RHS of (\ref{volume_comparison_1}), we observe
\begin{eqnarray*}
(\sqrt{cK})^{1/c}\int_0^{R/a}\bs_{-cK}^{1/c}(\tau)\d\tau &=& \int_0^{R/a}\left[\frac{1}{2}\left\{\exp(\sqrt{cK}\tau) - \exp(-\sqrt{cK}\tau)\right\}\right]^{1/c}\d\tau\\
%&\leq & \int_0^{\frac{R}{a}}\exp(\sqrt{cK}\tau)^{\frac{1}{c}}\d\tau \\
&\leq&\int_0^{R/a} \exp\left(\sqrt{\frac{K}{c}}\tau\right)\d\tau \\
%&=&\left[\sqrt{\frac{c}{K}}\exp\left(\sqrt{\frac{K}{c}}\tau\right) \right]_{0}^{\frac{R}{a}} \\
&=& \sqrt{\frac{c}{K}}\left\{\exp\left(\sqrt{\frac{K}{c}}\frac{R}{a}\right)-1\right\}.
\end{eqnarray*}
%Since
%\begin{equation*}
% \exp\left(\sqrt{\frac{K}{c}}\frac{R}{a}\right)-1 \leq \exp\left(\sqrt{\frac{K}{c}}\frac{R}{a}\right),
%\end{equation*}
Thus, we have
\begin{equation*}
\mu(B(x,R))\leq \mu(B(x,1))\frac{b}{a} \frac{1}{\int_0^{1/b} \bs_{-cK}(\tau)^{1/c} \,\d\tau}\sqrt{\frac{c}{K}}\exp\left(\sqrt{\frac{K}{c}}\frac{R}{a}\right)\frac{1}{(\sqrt{cK})^{1/c}}.
\end{equation*}
This implies
$$
\frac{\e^{p \alpha(R-1)} \mu\left(B(x,R)\right)}{\e^{p \alpha} \mu\left(B(x,1)\right)} \leq C_4 \exp \left(p\alpha R+ \sqrt{\frac{K}{c}}\frac{R}{a} \right) = C_4\exp(-\delta R) \rightarrow 0
$$
as $R \rightarrow \infty$, where $C_4$ is a constant depending on $c,a,b,K,\delta.$ Hence, (\ref{moto-15}) yields
$$
\lambda_{\mu,p}(M) \leq(1+\zeta)^{p-1}(-\alpha)^p .
$$
Since $\zeta>0$ and $\delta>0$ are arbitrary, this implies the theorem.
\end{proof}
When $m\in [n,\infty)$, $\eps = 1$ and $a = b = 1$, then $c = \frac{1}{m-1}$ and it holds
\begin{equation*}
\lambda_{\mu,p}(M)\leq \left(\frac{\sqrt{(m-1)K}}{p}\right)^p,
\end{equation*}
which recovers Theorem \ref{cheng-type-m-inequality}.
 
\section{Functional inequalities with $\eps$-range}
\subsection{Local Poincar\'{e} inequality}
In this subsection, we prove the following Poincar\'{e} inequality.
\begin{theorem}(Local Poincaré inequality)\label{poincare_thm}
Let $(M,g,\mu)$ be an $n$-dimensional complete weighted Riemannian manifold and $m \in (-\infty,1] \cup [n,+\infty]$,
$\ez \in \R$ in the $\ez$-range \eqref{epsilin-range}, $K > 0$ and $b \ge a>0$.
Assume that
\[ \Ric_\psi^m(v)\ge -K\e^{\frac{4(\ez-1)}{n-1}\psi(x)} g(v,v) \]
holds for all $v \in T_xM \setminus 0$ and
\begin{equation}
a \le \e^{-\frac{2(\ez-1)}{n-1}\psi} \le b.
\end{equation} Then we have
\begin{equation}\label{poincare_inequality}
\forall f \in C^{\infty}(M), \quad \int_{B}|f - f_B|^2 \d\mu \leq 2^{n+3}\left(\frac{2b}{a}\right)^{1/c}\exp\left(\sqrt{\frac{K}{c}}\frac{2r}{a}\right)r^2\int_{2B}|\nabla f|^2 \d\mu
\end{equation}
for all balls $B \subset M$ of radius $0<r<\infty$, where
\begin{equation*}
f_B := \frac{1}{\mu(B)}\int_B f \d\mu.
\end{equation*}
\end{theorem}
\begin{proof}
We apply the argument in \cite[Theorem 5.6.6, Lemma 5.6.7]{salof2}.
For any pair of points $(x, y) \in M \times M$, let
$$
\gamma_{x, y}:[0, d(x, y)] \rightarrow M%, \quad t \mapsto \gamma_{x, y}(t)
$$
be a geodesic from $x$ to $y$ parametrized by arclength.
We also set
\begin{equation*}
l_{x,y}(t) = \gamma_{x,y}(td(x,y))
\end{equation*}
for $t\in[0,1]$.
%ix a ball $B$ of radius $r$ and set
%\begin{equation*}
% f_B = \frac{1}{\mu(B)}\int_B f \d\mu.
%\end{equation*}
We have, using Jensen's inequality,
\begin{eqnarray*}
\int_B\left|f-f_B\right|^2 \d\mu %&=& \frac{1}{\mu(B)^2}\int_B\left|\int_B (f(x) - f(y)) \d\mu(y)\right|^2\d\mu(x)\\
&=& \int_B\left|\int_B (f(x) - f(y))\frac{\d\mu(y)}{\mu(B)}\right|^2\d\mu(x)\\
&\leq & \frac{1}{\mu(B)}\int_B\int_B |f(l_{x,y}(1)) - f(l_{x,y}(0))|^2 \d\mu(x)\ \d\mu(y)\\
&\leq & \frac{1}{\mu(B)}\int_B \int_B \left\{\int_0^{1} \left|\frac{\d(f\circ l_{x,y})}{\d t}(t)\right| \d t \right\}^2 \d\mu(x)\ \d\mu(y)\\
&\leq & \frac{1}{\mu(B)}\int_B\int_B \int_0^{1} \left|\frac{\d(f\circ l_{x,y})}{\d t}(t)\right|^2 \d t \ \d\mu(x)\ \d\mu(y)\\
&=&\frac{2}{\mu(B)} \int_B \int_B \int_{1 / 2}^{1}\left|\frac{\d(f\circ l_{x,y})}{\d t}(t)\right|^2 \d t \ \d\mu(x)\ \d\mu(y).
\end{eqnarray*}
To obtain the last equality we decompose the set
$$
\{(x, y, t): x, y \in B, t \in(0,1)\}
$$
into two pieces,
$$
\{(x, y, s): x, y \in B, t\in(1 / 2, 1)\}
$$
%%%%%%%%%%%%%%%
and
$$
\left\{(x, y, s): x, y \in B, t \in(0,1 / 2)\right\},
$$
then use $l_{x,y}(t) = l_{y,x}(1-t)$.
Now, suppose that we can bound the Jacobian $J_{x, t}$ of the map
$$
\Phi_{x, t}: y \mapsto l_{x,y}(t)
$$
from below by
\begin{equation}\label{moto-lem-5-6-7}
\forall x, y \in B, \forall s \in[1 / 2, 1], \quad J_{x, t}(y) \geq 1 / F(r),
\end{equation}
where $r$ is the radius of the ball $B$. Then
\begin{eqnarray*}
\int_B\int_B\int_{1/2}^1 \left|\frac{\d f(l_{x,y}(t))}{\d t}\right|^2 \d t\  \d\mu(x)\ \d\mu(y) &\leq& F(r)\int_B\int_B\int_{1/2}^1\left|\frac{\d f(l_{x,y}(t))}{\d t}\right|^2J_{x,t}(y)\d t\ \d\mu(x)\ \d\mu(y)\\
%&=& F(r)\int_B\int_B\int_{0}^1\left|\nabla f(l_{x,y}(t)) \frac{dl}{\d t}\right|^2J_{x,t}(y)\d t\ \d\mu(x)\ \d\mu(y)\\
&\leq& F(r)\int_B\int_B\int_{0}^1\left|\nabla f(l_{x,y}(t))\right|^2 d(x,y)^2 J_{x,t}(y)\d t \ \d\mu(x)\ \d\mu(y)\\
&\leq& (2r)^2F(r) \int_0^1\int_B \int_B \left|\nabla f(l_{x,y}(t)) \right|^2 J_{x,t}(y) \d\mu(y) \ \d\mu(x)\ \d t\\
&=& (2r)^2 F(r)\int_0^1\left[ \int_B \left(\int_{\Phi_{x,t}(B)} \left| \nabla f(z) \right|^2 \d\mu(z) \right)\d\mu(x)\right] \d t\\
&\leq & (2r)^2 F(r) \int_0^1\left[ \int_B \left(\int_{2B} \left| \nabla f(z) \right|^2 \d\mu(z) \right)\d\mu(x)\right] \d t\\
&\leq & (2r)^2 F(r) \mu(B) \int_{2B} |\nabla f(z) |^2 \d\mu(z).
\end{eqnarray*}
 
We finally prove (\ref{moto-lem-5-6-7}).
Let $\xi$ be the unit tangent vector at $x$ such that $\left.\partial_s \gamma_{x, y}(s)\right|_{s=0}=\xi$. Let $I(x, s, \xi)$ be the Jacobian of the map $\exp_x : T_xM\rightarrow M$ at $s\xi$ with respect to $\mu$. Then
$$
\d\mu=I(x, s, \xi) \d s \ \d \xi,
$$
where $\d \xi$ is the usual measure on the sphere. Using the notation in Subsection \ref{epsilon-range-section}, we have $I(x,s,\xi) = \e^{-\psi(\eta(s))}h_0^{n-1}(s)$.
According to (\ref{non_increasing_property}), we find that
$$
s \rightarrow \frac{I(x, s, \xi)}{\bs_{-cK}(\varphi_\eta(s))^{1/c}}
$$
is non-increasing. Under the relationship $l_{x,y}(t) = \gamma_{x,y}(s)$, it follows that
$$
J_{x, t}(y)= \left(\frac{s}{d(x,y)}\right)^n \frac{I(x, s, \xi)}{I(x, d(x, y), \xi)} \geq \left(\frac{1}{2}\right)^n\frac{\bs_{-cK}(\varphi_\eta(s))^{1/c}}{\bs_{-cK}(\varphi_\eta(d(x,y)))^{1/c}}
$$
for all $s \in(d(x,y)/2, d(x, y))$.
Thus, we have, since $s/b \leq \varphi_\eta(t)\leq s/a$,
\begin{eqnarray*}
J_{x,t}(y) &\geq&\left(\frac{1}{2}\right)^n \frac{\bs_{-cK}(\varphi_\eta(d(x,y)/2))^{1/c}}{\bs_{-cK}(\varphi_\eta(d(x,y)))^{1/c}}\\
&\geq&\left(\frac{1}{2}\right)^n \left(\frac{\varphi_\eta(d(x,y)/2)}{\varphi_\eta(d(x,y))}\right)^{1/c}\exp\left(-\sqrt{\frac{K}{c}}\varphi_\eta(d(x,y))\right)\\
&\geq& \left(\frac{1}{2}\right)^n\left(\frac{ad(x,y)/2}{bd(x,y)}\right)^{1/c}\exp\left(-\sqrt{\frac{K}{c}}\frac{d(x,y)}{a}\right)\\
&\geq& \left(\frac{1}{2}\right)^n\left(\frac{a}{2b}\right)^{1/c}\exp\left(-\sqrt{\frac{K}{c}}\frac{2r}{a}\right).
\end{eqnarray*}
This proves (\ref{moto-lem-5-6-7}) with $F(r) = \left\{\left(\frac{1}{2}\right)^n\left(\frac{a}{2b}\right)^{1/c}\exp\left(-\sqrt{\frac{K}{c}}\frac{2r}{a}\right)\right\}^{-1}$ and the theorem follows.
\end{proof}

Given that we have the local Poincar\'{e} inequality and the volume doubling property (obtained explicitly later in \eqref{volume-doubling-property}), we can apply \cite[Corollary 5.3.5]{salof2} and we obtain the following inequality.
%We remark that we have the following inequality since we can apply \cite[Corollary 5.3.5]{salof2} because we have the volume doubling property (obtained explicitly later in \eqref{volume-doubling-property}) and \eqref{poincare_inequality}
%We note that we do not use this inequality in the following arguments.
\begin{corollary}\textit{
    Under the same assumptions as in Theorem \ref{poincare_thm}, there exist constants $C,P$ such that
    \begin{equation*}
        \forall f \in C^{\infty}(M), \quad \int_B |f - f_B|^2 \d\mu \leq P\e^{Cr} r^2 \int_B |\nabla f|^2 \d\mu
    \end{equation*}
    for all balls $B\subset M$ of radius $r>0$.}
\end{corollary}
 
\subsection{Local Sobolev inequality}\label{Sobolev_section}
It is shown in \cite{salof1} that the volume doubling property and Poincar\'{e} inequality imply a local Sobolev inequality. We follow this line with Theorems \ref{bishop_gromov} and \ref{poincare_thm}.
\begin{theorem}(Local Sobolev inequality)\label{sobolev_thm}
Let $(M,g,\mu)$ be an $n$-dimensional complete weighted Riemannian manifold with $n \geq 3$ and $m \in (-\infty,1] \cup [n,+\infty]$,
$\ez \in \R$ in the $\ez$-range \eqref{epsilin-range}, $K > 0$ and $b \ge a>0$.
Assume that
\[ \Ric_\psi^m(v)\ge -K\e^{\frac{4(\ez-1)}{n-1}\psi(x)} g(v,v) \]
holds for all $v \in T_xM \setminus 0$ and
\begin{equation}\label{measure_pinching}
a \le \e^{-\frac{2(\ez-1)}{n-1}\psi} \le b.
\end{equation} Then there exist constants $D, E$ depending on $c,a,b, n$ such that for all $B(o,r)\subset M$ we have for $f\in C_0^{\infty}(B(o,r))$,
\begin{equation*}
\left(\mu(B(o,r))^{-1}\int_{B(o,r)}|f|^{\frac{2(1 + c)}{1-c}}\d\mu\right)^{\frac{1-c}{1 + c}}\leq E\exp\left(D\left(1 + \sqrt{\frac{K}{c}}\right)\frac{r}{a}\right)r^2 \mu(B(o,r))^{-1}\int_{B(o,r)}(|\nabla f|^2 + r^{-2} f^2)\d\mu.
\end{equation*}
\end{theorem}
We first prove two lemmas. We set
\begin{equation*}
f_s(x) = \int \chi_s(x,z)f(z)\d\mu(z),
\end{equation*}
where $V(x,s) = \mu(B(x,s))$ and $\chi_s(x,z) = \frac{1}{V(x,s)}1_{B(x,s)}(z)$.
\begin{lemma}\label{lemma_1}
Under the same assumptions as in Theorem \ref{sobolev_thm},
there exists a constant $C_5$ %depending on $c,a,b,K,r$
such that for all $y\in M$ and all $0 < s < r$, we have
\begin{equation*}
\|f_s\|_2 \leq C_5V^{-\frac{1}{2}}\left(\frac{r}{s}\right)^{\frac{1}{2}\left(1 + \frac{1}{c}\right)}\|f\|_1,
\end{equation*}
for all $f\in C^{\infty}_0(B)$, where $B = B(y,r)$ and $V = V(r) = V(y,r)$.
\end{lemma}
\begin{proof}
We apply the argument in \cite[Lemma 2.3]{salof1}.
We use the notations in Subsection \ref{epsilon-range-section}. %For $0 \leq \alpha \leq \beta$, %let
For $\tau \geq 0$, $0 < s < r$, we set
\begin{equation*}
t := \frac{r}{a}\frac{b}{s}\tau.
\end{equation*}
Since $\frac{r}{a}\frac{b}{s} \geq 1$, we have $\tau \leq t$.
Hence, by direct computations, we obtain
\begin{equation*}
\bs_{-cK}(t)^{1/c} \leq \bs_{-cK}(\tau)^{1/c}\left(\frac{t}{\tau}\right)^{1/c}\exp\left(\sqrt{\frac{K}{c}}t\right) .
\end{equation*}
Integrating both sides in $t$ from $0$ to $r/a$, we have
\begin{eqnarray*}
\int_0^{r/a}\bs_{-cK}(t)^{1/c} \d t &\leq& \int_0^{r/a}\bs_{-cK}(\tau)^{1/c}\left(\frac{t}{\tau}\right)^{1/c}\exp\left(\sqrt{\frac{K}{c}}t\right) \d t\\
&\leq& \left(\frac{r}{a}\frac{b}{s}\right)^{1/c}\exp\left(\sqrt{\frac{K}{c}}\frac{r}{a}\right)\int_0^{r/a}\bs_{-cK}(\tau)^{1/c}\d t\\
&=& \left(\frac{r}{a}\frac{b}{s}\right)^{1 + \frac{1}{c}}\exp\left(\sqrt{\frac{K}{c}}\frac{r}{a}\right)\int_0^{s/b}\bs_{-cK}(\tau)^{1/c}\d \tau.
\end{eqnarray*}
According to Theorem \ref{bishop_gromov}, we find
\begin{eqnarray*}
\frac{V(r)}{V(s)} &\leq& \frac{b}{a}\frac{\int_0^{r/a}\bs_{-cK}(\tau)^{1/c} \d \tau}{\int_0^{s/b}\bs_{-cK}(\tau)^{1/c}\d \tau}\\
&\leq & \left(\frac{b}{a}\right)^{2 + \frac{1}{c}}\left(\frac{r}{s}\right)^{1 + \frac{1}{c}}\exp\left(\sqrt{\frac{K}{c}}\frac{r}{a}\right).
\end{eqnarray*}
%Hence,
%\begin{equation*}
% V(r)\leq V(s)\left(\frac{b}{a}\right)^{2 + \frac{1}{c}}\left(\frac{r}{s}\right)^{1 + \frac{1}{c}}\exp\left(\sqrt{\frac{K}{c}}\frac{r}{a}\right).
%\end{equation*}
%%%%%%%%%%%%%%%%%%%%%
%\begin{eqnarray*}
% \frac{V(r)}{b} &=& \frac{1}{b}\int_{U_yM}\int_0^r h(t)^\frac{1}{c} \d t\ dv\\
% &\leq& \int_0^{\varphi(r)}h_1(\tau)^{\frac{1}{c}}\d\tau\\
% &\leq& \int_{U_yM}\int_0^{\varphi(r)}h_1(\tau)^{\frac{1}{c}}\d\tau\ dv\\
% &\leq& \left(\frac{\varphi(r)}{\varphi(s)}\right)^{\frac{1}{c} + 1}\exp\left(\sqrt{\frac{K}{c}}\varphi(r)\right) \int_{U_yM}\int_0^{\varphi(s)}h_1(\tau)^{\frac{1}{c}}\d\tau\\
% &\leq & \left(\frac{r/a}{s/b}\right)^{1 + \frac{1}{c}}\exp\left(\sqrt{\frac{K}{c}}\frac{r}{a}\right)\frac{V(s)}{a}
%\end{eqnarray*}
%%%%%%%%%%%%%%%%%%%%%
Therefore, we also have the doubling property:
\begin{equation}\label{volume-doubling-property}
V(2r)\leq V(r)\left(\frac{b}{a}\right)^{2 + \frac{1}{c}}2^{\frac{1}{c}+ 1}\exp\left(\sqrt{\frac{K}{c}}\frac{2r}{a}\right).
\end{equation}
For $x,z\in M$ satisfying $d(z,x) < s$, we have
\begin{eqnarray*}
V(z,s)&\leq& V(x,2s)\\
&\leq & V(x,s)\left(\frac{b}{a}\right)^{2 + \frac{1}{c}}2^{\frac{1}{c}+ 1}\exp\left(\sqrt{\frac{K}{c}}\frac{2s}{a}\right).
\end{eqnarray*}
This implies
\begin{equation*}
\chi_s(x,z)\leq \left(\frac{b}{a}\right)^{2 + \frac{1}{c}}2^{\frac{1}{c}+ 1}\exp\left(\sqrt{\frac{K}{c}}\frac{2s}{a}\right) \chi_s(z,x).
\end{equation*}
Thus,
\begin{equation}\label{estimate_1}
\|f_s\|_1 \leq \left(\frac{b}{a}\right)^{2 + \frac{1}{c}}2^{\frac{1}{c}+ 1}\exp\left(\sqrt{\frac{K}{c}}\frac{2s}{a}\right)\|f\|_1.
\end{equation}
We moreover assume $B\cap B(x,s)\neq\emptyset$.
Since
\begin{equation*}
\frac{V(x,2r + s)}{V(x,s)}\leq \left(\frac{b}{a}\right)^{2 + \frac{1}{c}}\left(\frac{2r + s}{s}\right)^{1 + \frac{1}{c}}\exp\left(\sqrt{\frac{K}{c}}\frac{2r + s}{a}\right)
\end{equation*}
and
\begin{equation*}
\frac{V(x,4r)}{V(x,2r + s)}\leq \left(\frac{b}{a}\right)^{2 + \frac{1}{c}}\left(\frac{4r}{2r + s}\right)^{1 + \frac{1}{c}}\exp\left({\sqrt{\frac{K}{c}}}\frac{4r}{a}\right),
\end{equation*}
we have
\begin{eqnarray*}
\frac{1}{V(x,s)}&\leq& \left(\frac{b}{a}\right)^{2 + \frac{1}{c}}\left(\frac{2r + s}{s}\right)^{1 + \frac{1}{c}}\exp\left(\sqrt{\frac{K}{c}}\frac{2r + s}{a}\right)\frac{1}{V(x,2r + s)}\\
&\leq & \left(\frac{b}{a}\right)^{2\left(2 + \frac{1}{c}\right)}\left(\frac{2r + s}{s}\right)^{1 + \frac{1}{c}}\exp\left(\sqrt{\frac{K}{c}}\frac{2r + s}{a}\right)\left(\frac{4r}{2r + s}\right)^{1 + \frac{1}{c}}\exp\left(\sqrt{\frac{K}{c}}\frac{4r}{a}\right)\frac{1}{V(x,4r)}\\
&\leq& \left(\frac{b}{a}\right)^{2\left(2 + \frac{1}{c}\right)}\left(\frac{4r}{s}\right)^{1 + \frac{1}{c}}\exp\left(\sqrt{\frac{K}{c}}\frac{6r + s}{a}\right)\frac{1}{V(y,r)}.
\end{eqnarray*}
Hence,
\begin{eqnarray*}
\|f_s\|_{\infty} &=& \left\|\int\chi_s(x,z)f(z)\d\mu(z)\right\|_{\infty}\\
&\leq& \left(\frac{b}{a}\right)^{2\left(2 + \frac{1}{c}\right)}\left(\frac{4r}{s}\right)^{1 + \frac{1}{c}}\exp\left(\sqrt{\frac{K}{c}}\frac{6r + s}{a}\right)\frac{\|f\|_1}{V(y,r)}.
\end{eqnarray*}
Using (\ref{estimate_1}), we have
\begin{eqnarray*}
\|f_s\|_2 &=& \left(\int f_s^2\d\mu\right)^{\frac{1}{2}}\\
&\leq & \sqrt{\|f_s\|_{\infty}}\sqrt{\|f_s\|_1}\\
&\leq& \left(\frac{b}{a}\right)^{2 + \frac{1}{c}}\left(\frac{4r}{s}\right)^{\frac{1}{2}\left(1 + \frac{1}{c}\right)}\exp\left(\sqrt{\frac{K}{c}}\frac{6r + s}{2a}\right)\frac{1}{\sqrt{V(y,r)}}\left(\frac{b}{a}\right)^{\frac{1}{2}\left(2 + \frac{1}{c}\right)}2^{\frac{1}{2}\left(\frac{1}{c}+ 1\right)}\exp\left(\sqrt{\frac{K}{c}}\frac{s}{a}\right)\|f\|_1\\
&=& \left(\frac{b}{a}\right)^{3 + \frac{3}{2c}}2^{\frac{1}{2}\left(1 + \frac{1}{c}\right)}\left(\frac{4r}{s}\right)^{\frac{1}{2}\left(1 + \frac{1}{c}\right)}\exp\left(\sqrt{\frac{K}{c}}\frac{6r+3s}{2a}\right)\frac{1}{\sqrt{V(y,r)}}\|f\|_1.
\end{eqnarray*}
Setting $C_5 = \left(\frac{b}{a}\right)^{3 + \frac{3}{2c}}2^{\frac{1}{2}\left(1 + \frac{1}{c}\right)}4^{\frac{1}{2}\left(1 + \frac{1}{c}\right)}\exp\left(\sqrt{\frac{K}{c}}\frac{9r}{2a}\right)$, we get the desired inequality.
\end{proof}
%\begin{remark}\label{gromov_salof}
% We refer to \cite[Lemma 5.6.7]{salof2} for (\ref{ineq_gromov}) and \cite[Theorem 5.6.4]{salof2} for (\ref{gromov_eq_epsilon_range}).
%\end{remark}
 
\begin{lemma}\label{lemma_2}
We fix a constant $r > 0$.
Under the same assumptions as in Theorem \ref{sobolev_thm},
there exists $C_6$ depending only on $c,a,b,r,K,n$ such that
\begin{equation*}
\|f - f_s\|_2\leq C_6s \|\nabla f\|_2,\quad f\in C_0^{\infty}(M)
\end{equation*}
for all $0 < s < r$.
\end{lemma}
\begin{proof}
We apply the argument in \cite[Lemma 2.4]{salof1}.
Fix $a > 0$, let $\{B_j:\ j\in J\}$ be a collection of balls of radius $s/2$ such that $B_i\cap B_j = \emptyset$ if $i\neq j$ and $M = \bigcup_{i\in J}2B_i$. %, where $tB = B(x,tr)$ if $B = B(x,r)$.
For $z\in M$, let $J(z) = \{i\in J:\ z\in 8B_i\}$ and $N(z) = \#J(z)$. %is bounded from above by $N_0$ which depends only on the volume doubling constant
%$\left(\frac{b}{a}\right)^{2 + \frac{1}{c}}2^{\frac{1}{c}+ 1}\exp\left(\sqrt{\frac{K}{c}}\frac{2r}{a}\right)$
%in (\ref{volume-doubling-property}).
We first estimate $N(z)$ from above.
Let $B_z$ be a ball in $\{B_j:\ j\in J\}$ such that $z\in 2B_z$.
For $i\in J(z)$, we have $B_z \subset 16 B_i$. Hence,
\begin{equation*}
\mu(B_z)\leq \mu(16B_i) \leq C_7^4\mu(B_i),
\end{equation*}
where
\begin{equation*}
C_7 := \left(\frac{b}{a}\right)^{2 + \frac{1}{c}}2^{\frac{1}{c}+ 1}\exp\left(\sqrt{\frac{K}{c}}\frac{8r}{a}\right) \geq \left(\frac{b}{a}\right)^{2 + \frac{1}{c}}2^{\frac{1}{c}+ 1}\exp\left(\sqrt{\frac{K}{c}}\frac{8s}{a}\right).
\end{equation*}
Therefore, we have
\begin{equation*}
\sum_{i\in J(z)}\mu(B_i) \geq N(z)\frac{\mu(B_z)}{C_7^4}.
\end{equation*}
On the other hand, for $i\in \{j\in J:\ z\in 8B_j\}$, we have $B_i \subset 16 B_z$. Hence,
\begin{equation*}
\sum_{i\in J(z)}\mu(B_i) \leq \mu(16 B_z) \leq C_7^4\mu(B_z).
\end{equation*}
Therefore, we find
\begin{equation*}
N(z)\frac{\mu(B_z)}{C_7^4} \leq C_7^4\mu(B_z).
\end{equation*}
Letting $N_0 := C_7^8$, we have $N(z)\leq N_0$.
%We have
%\begin{equation*}
% N(z) \leq C^8 = \left\{\left(\frac{b}{a}\right)^{2 + \frac{1}{c}}2^{\frac{1}{c}+ 1}\exp\left(\sqrt{\frac{K}{c}}\frac{r}{a}\right)\right\}^8
%\end{equation*}
We now estimate $\|f - f_s\|_2$. Note that
\begin{equation}\label{star_3}
\|f - f_s\|_2^2\leq \sum_{i\in J}\left(2\int_{2B_i}|f(x) - f_{4B_i}|^2 + |f_{4B_i} - f_s(x)|^2\ \d\mu(x)\right).
\end{equation}
By the Poincar\'{e} inequality (\ref{poincare_inequality}), we have
\begin{equation}\label{estimate_lem3_1}
\int_{4B_i} |f(x) - f_{4B_i}|^2 \d\mu(x) \leq 2^{n+3}\left(\frac{2b}{a}\right)^{\frac{1}{c}}\exp\left(\sqrt{\frac{K}{c}}\frac{4s}{a}\right)(2s)^2\int_{8B_i}|\nabla f|^2 \d\mu \leq C_8s^2\int_{8B_i}|\nabla f|^2 \d\mu,
\end{equation}
where
\begin{equation*}
C_8 = 2^{n+5} \left(\frac{2b}{a}\right)^{\frac{1}{c}}\exp\left(\sqrt{\frac{K}{c}}\frac{4r}{a}\right).
\end{equation*}
Since for any $ x \in 2B_i = B(x_i,s),$
\begin{equation*}
V(x_i,s)\leq V(x,2s) \leq V(x,s)\left(\frac{b}{a}\right)^{2 + \frac{1}{c}}2^{\frac{1}{c}+ 1}\exp\left(\sqrt{\frac{K}{c}}\frac{2s}{a}\right),
\end{equation*}
we have
%\begin{eqnarray}
% \int_{2B_i}|f_{4B_i} - f_s(x)|^2 &\leq& \int_{2B_i}\int \chi_s(x,z)|f_{4B_i} - f(z)|^2 \d\mu(z)\d\mu(x)\nonumber\\
% &\leq& \frac{1}{V(x_i,s)}\left(\frac{b}{a}\right)^{2 + \frac{1}{c}}2^{\frac{1}{c}+ 1}\exp\left(\sqrt{\frac{K}{c}}\frac{2s}{a}\right)\int_{2B_i}\int_{4B_i}|f_{4B_i} - f(z)|^2\d\mu(z)\d\mu(x)\nonumber\\
% &\leq& \left(\frac{b}{a}\right)^{2 + \frac{1}{c}}2^{\frac{1}{c}+ 1}\exp\left(\sqrt{\frac{K}{c}}\frac{2s}{a}\right)C_2s^2\int_{8B_i}|\nabla f|^2\label{star_4}
%\end{eqnarray}
%%%%%%%%%%%%%%%%%%%%%%%%%%%%
\begin{eqnarray}
\int_{2B_i}\left|f_{4B_i} - f_s(x)\right|^2 \d\mu(x) %&=& \int_{2B_i}\left|\int \chi_s(x,z) \d\mu(z)f_{4B_i}(x) - \int \chi_s(x,z)f(z)\d\mu(z)\right|^2\d\mu(x)\nonumber\\
&=& \int_{2B_i}\left|\int_{B(x,s)}\frac{1}{V(x,s)}\{f_{4B_i} - f(z)\}\ \d\mu(z)\right|^2\d\mu(x)\nonumber\\
&\leq & \int_{2B_i}\int_{B(x,s)} \frac{1}{V(x,s)} \left|f_{4B_i} - f(z) \right|^2\d\mu(z) \ \d\mu(x)\nonumber\\
&\leq&\frac{1}{V(x_i,s)}\left(\frac{b}{a}\right)^{2 + \frac{1}{c}}2^{\frac{1}{c}+ 1}\exp\left(\sqrt{\frac{K}{c}}\frac{2s}{a}\right) \int_{2B_i}\int_{4B_i}|f_{4B_i} - f(z)|^2\d\mu(z)\ \d\mu(x)\nonumber\\
&\leq & \left(\frac{b}{a}\right)^{2 + \frac{1}{c}}2^{\frac{1}{c}+ 1}\exp\left(\sqrt{\frac{K}{c}}\frac{2r}{a}\right)C_8s^2\int_{8B_i}|\nabla f|^2\d\mu.\label{star_4}
\end{eqnarray}
%%%%%%%%%%%%%%%%%%%%%%%%%%%%
Using (\ref{star_3}), (\ref{estimate_lem3_1}), (\ref{star_4}), we have
\begin{equation*}
\|f - f_s\|^2_2\leq C_{9}s^2\sum_{i\in J}\int_{8B_i}|\nabla f|^2 \d\mu \leq C_{9}N_0s^2\|\nabla f\|^2_2,
\end{equation*}
where
\begin{equation*}
C_{9} = 4\left(\frac{b}{a}\right)^{2 + \frac{1}{c}}2^{\frac{1}{c}+ 1}\exp\left(\sqrt{\frac{K}{c}}\frac{2r}{a}\right)C_8 \geq 2\left\{\left(\frac{b}{a}\right)^{2 + \frac{1}{c}}2^{\frac{1}{c}+ 1}\exp\left(\sqrt{\frac{K}{c}}\frac{2r}{a}\right)C_8 + C_8\right\}.
\end{equation*}
Therefore, setting
\begin{eqnarray*}
C_6 &:=&\sqrt{N_0C_{9}} \\
&=& \left\{\left(\frac{b}{a}\right)^{2 + \frac{1}{c}}2^{\frac{1}{c}+ 1}\exp\left(\sqrt{\frac{K}{c}}\frac{8r}{a}\right) \right\}^4 \times \sqrt{4\left(\frac{b}{a}\right)^{2 + \frac{1}{c}}2^{\frac{1}{c}+ 1}\exp\left(\sqrt{\frac{K}{c}}\frac{2r}{a}\right)}\times \sqrt{2^{n + 5} \left(\frac{2b}{a}\right)^{\frac{1}{c}}\exp\left(\sqrt{\frac{K}{c}}\frac{4r}{a}\right)}\\
&=& 2^{8 + \frac{5}{c} + \frac{n}{2}}\left(\frac{b}{a}\right)^{9 + \frac{5}{c}}\exp\left(\sqrt{\frac{K}{c}}\frac{35r}{a}\right),
\end{eqnarray*}
we have the desired inequality.
\end{proof}
 
\underline{\textit{Proof of Theorem \ref{sobolev_thm}}}
 
We apply the argument in \cite[Theorem 2.1]{salof1}. Fix $x\in M, r > 0$. For $0 < s \leq r$ and $f\in C_0^{\infty}(B(x,r))$, we have
\begin{equation*}
\|f\|_2\leq \|f - f_s\|_2 + \|f_s\|_2.
\end{equation*}
It follows from Lemmas \ref{lemma_1}, \ref{lemma_2} that
\begin{equation*}
\|f\|_2\leq C_6 s\|\nabla f\|_2 + C_5V^{-\frac{1}{2}}\left(\frac{r}{s}\right)^{\frac{\nu}{2}}\|f\|_1,
\end{equation*}
where $\nu = 1 + \frac{1}{c}$.
Hence, we obtain
\begin{equation}\label{robinson}
\|f\|_2\leq 4C_6s\left(\|\nabla f\|_2 + \frac{1}{r}\|f\|_2\right) + C_5V^{-\frac{1}{2}}\left(\frac{r}{s}\right)^{\frac{\nu}{2}}\|f\|_1.
\end{equation}
To obtain the minimum of the RHS of (\ref{robinson}), we consider its differential with respect to $s > 0$. At $s > 0$ which attains the minimum, we have
\begin{equation*}
4C_6\left(\|\nabla f\|_2 + \frac{1}{r}\|f\|_2\right) + C_5V^{-\frac{1}{2}}r^{\frac{\nu}{2}}\left(-\frac{\nu}{2}\right)s^{-\frac{\nu}{2}-1}\|f\|_1 = 0.
\end{equation*}
Thus,
\begin{equation}\label{value_of_s}
s^{\frac{\nu}{2} + 1} = C_{10}\frac{V^{-\frac{1}{2}}r^{\frac{\nu}{2}}\|f\|_1}{\|\nabla f\|_2 + \frac{1}{r}\|f\|_2},
\end{equation}
where $C_{10} = \frac{\nu}{2}\frac{C_5}{4C_6}$.
Substituting (\ref{value_of_s}) to the RHS of (\ref{robinson}), we obtain
\begin{eqnarray*}
\|f\|_2&\leq&4C_6\left\{C_{10}\frac{V^{-\frac{1}{2}}r^{\frac{\nu}{2}}\|f\|_1}{\|\nabla f\|_2 + \frac{1}{r}\|f\|_2}\right\}^{\frac{2}{2 + \nu}}\left(\|\nabla f\|_2 + \frac{1}{r}\|f\|_2\right) + C_5V^{-\frac{1}{2}}r^{\frac{\nu}{2}}\left\{\left(C_{10}\frac{V^{-\frac{1}{2}}r^{\frac{\nu}{2}}\|f\|_1}{\|\nabla f\|_2 + \frac{1}{r}\|f\|_2}\right)^{\frac{2}{2 + \nu}}\right\}^{-\frac{\nu}{2}}\|f\|_1\\
&=&4C_6C_{10}^{\frac{2}{2 + \nu}}\left(\|\nabla f\|_2 + \frac{1}{r}\|f\|_2\right)^{-\frac{2}{2 + \nu} + 1}\left(V^{-\frac{1}{2}}r^{\frac{\nu}{2}}\|f\|_1\right)^{\frac{2}{2 + \nu}} \\
&&\qquad + C_5C_{10}^{\frac{-\nu}{2 + \nu}}\left(\|\nabla f\|_2 + \frac{1}{r}\|f\|_2\right)^{\frac{\nu}{2 + \nu}}V^{-\frac{1}{2}}r^{\frac{\nu}{2}}\left(V^{-\frac{1}{2}}r^{\frac{\nu}{2}}\|f\|_1\right)^{-\frac{\nu}{2 + \nu}}\|f\|_1\\
&=& \left(\|\nabla f\|_2 + \frac{1}{r}\|f\|_2\right)^{\frac{\nu}{2 + \nu}}\left\{4C_6C_{10}^{\frac{2}{2 + \nu}}V^{-\frac{1}{2}\left(\frac{2}{2 + \nu}\right)}r^{\frac{\nu}{2}\left(\frac{2}{2 + \nu}\right)}\|f\|_1^{\frac{2}{2 + \nu}} \right.\\
&&\qquad \left.+ C_5C_{10}^{\frac{-\nu}{2 + \nu}}V^{-\frac{1}{2}\left(1-\frac{\nu}{2 + \nu}\right)}r^{\frac{\nu}{2}\left(1 - \frac{\nu}{2 + \nu}\right)}\|f\|_1^{1 - \frac{\nu}{2 + \nu}}\right\}\\
&=& \left\{4C_6C_{10}^{\frac{2}{2 + \nu}} + C_5C_{10}^{\frac{-\nu}{2 + \nu}}\right\}\left(\|\nabla f\|_2 + \frac{1}{r}\|f\|_2\right)^{\frac{\nu}{2 + \nu}}V^{-\frac{1}{2}\left(\frac{2}{2 + \nu}\right)}r^{\frac{\nu}{2}\left(\frac{2}{2 + \nu}\right)}\|f\|_1^{\frac{2}{2 + \nu}}.
\end{eqnarray*}
Hence
\begin{equation*}
\|f\|_2^{2 + \frac{4}{\nu}}\leq \left\{4C_6C_{10}^{\frac{2}{2 + \nu}} + C_5C_{10}^{\frac{-\nu}{2 + \nu}}\right\}^{2 + \frac{4}{\nu}} V^{-\frac{2}{\nu}}r^2 \left\{2\left(\|\nabla f\|_2^2 + \frac{\|f\|^2_2}{r^2}\right)\right\}\|f\|_1^{\frac{4}{\nu}}.
\end{equation*}
Recalling the expressions of $C_5$ and $C_6$ in the proofs of Lemmas \ref{lemma_1}, \ref{lemma_2},
we choose constants $E_1, E_2$ depending on $c,a,b,n$ such that
\begin{equation*}
C_6C_{10}^{\frac{2}{2 + \nu}} = E_1\exp\left(\sqrt{\frac{K}{c}}\left(\frac{35r}{a} - \frac{61r}{2a}\frac{2}{2 + \nu}\right)\right)
\end{equation*}
and
\begin{equation*}
C_5C_{10}^{\frac{-\nu}{2 + \nu}} = E_2\exp\left(\sqrt{\frac{K}{c}}\left(\frac{9r}{2a} - \frac{61r}{2a}\frac{-\nu}{2 + \nu}\right)\right).
\end{equation*}
Thus, there exist constants $D,E_3$ such that
\begin{equation*}
\left\{4C_6C_{10}^{\frac{2}{2 + \nu}} + C_5C_{10}^{\frac{-\nu}{2 + \nu}}\right\}^{2 + \frac{4}{\nu}} < E_3\exp\left(D\left(1 + \sqrt{\frac{K}{c}}\right)\frac{r}{a}\right).
\end{equation*}
We remark that $D,E_3$ depend only on $c,a,b,n$.
Since
$c \leq \frac{1}{n-1}$,
we have $\nu = 1 + \frac{1}{c} \geq n > 2$ when $n \geq 3$. Hence, we can use Theorem \ref{proof_ingredients_analysis} and Theorem \ref{sobolev_thm} follows.
\qed
%\begin{remark}
% It is shown in \cite{salof1} that volume doubling property and Poincar\'{e} inequality imply a local Sobolev inequality.
%we observed proofs of Lemma \ref{lemma_1}, Lemma \ref{lemma_2} and Theorem 8 for completeness.
%\end{remark}
\begin{remark}\label{last}
At the end of the proof of Theorem \ref{sobolev_thm}, we use the fact that $1 + \frac{1}{c} > 2$, it is the only reason why we need the assumption $n\geq 3$.
%Although most of the arguments are seems to be experts-well-known (For $m\in (n,\infty)$, we refer \cite{linfengwang2} and for $m = \infty$, we refer \cite{soliton}, for example), it is interesting that we have the property ($1 + \frac{1}{c} > 2$) and local Sobolev inequality still holds in $\eps$-range.
In the case of $n = 2$, we have the local Sobolev inequality when $\eps\neq 0$ under the same curvature bound and (\ref{measure_pinching}) in Theorem \ref{sobolev_thm}.
\end{remark}
%By completely the same argument as in \cite[Theorem 3.1]{soliton}, we have a gradient estimate of the harmonic function as follows.
%\begin{theorem}
% Under the assumptions of Theorem \ref{sobolev_thm}, there exists a constant $C(n,N,\eps)$ such that for any $u > 0$ with $\Delta_\psi u = 0$, we have
% \begin{equation*}
% |\nabla \log u|\leq C(n,m,\eps,a,b).
% \end{equation*}
%\end{theorem}
 
\begin{remark}
One of the possible subjects of further research is the gradient estimate of eigenfunctions of the weighted Laplacian with $\eps$-range, which turned out to be difficult. If $\Ric_\psi^m$ is bounded from below with $m > n$, then one way to obtain the gradient estimate is to apply the Li-Yau trick as described in \cite{wu}, \cite{wu2} and another way is to use the DeGiorgi-Nash-Moser theory \cite{li3} as described in \cite{soliton}. %We also remark that a Liouville type theorem is obtained by Moser's iteration argument in \cite{linfengwang2}.
Once we obtained the gradient estimate by the Li-Yau trick, an upper bound of eigenvalues of the weighted Laplacian is obtained as in \cite{wu}, \cite{wu2}.
However, it seems that the Li-Yau trick and Moser's iteration argument in \cite{linfengwang2} do not work well in the case where $\Ric_\psi^m$ is bounded from below with $m \leq 1$. The main difficulty stems from the lack of a suitable Bochner formula for analyzing lower Bakry-\'{E}mery-Ricci curvature bounds with $\eps$-range. Although \cite[Lemma 2.1]{kuwae_sakurai} obtained the Bochner formula for the distance function with $\eps$-range, a suitable Bochner formula for eigenfunctions of the weighted Laplacian is yet to be known.
Finding a suitable Bochner formula for eigenfunctions is our future work.
 
\end{remark}

\section*{Appendix: Upper bound of the $L^p$-spectrum for deformed measures}
Although we considered the Riemannian distance $d$, it is also possible to study comparison theorems associated with a metric deformed by using the weight function
(we refer \cite{yoroshikin}, \cite{kuwae_li}, \cite{kuwae_sakurai}, for example).
In this appendix, we start from a volume comparison theorem in \cite{kuwae_sakurai} and prove a variant of Cheng type inequality for the $L^p$-spectrum.
 
Let $(M,g,\mu = \e^{-\psi}v_g)$ be an $n$-dimensional weighted Riemannian manifold, $m\in (-\infty,1]\cup [n,+\infty]$ and $\eps \in \mathbb{R}$ in the range (\ref{epsilin-range}).
We fix a point $q\in M$. % and assume $\psi(q) = 0$. 
We define lower semi continuous functions $s_q:M\to \mathbb{R}$ by
\begin{equation*}
\quad s_q(x):=\inf_{\gamma}\int^{d(q,x)}_0\,\e^{ -\frac{2(1-\eps)\psi(\gamma(\xi))}{n-1} }\,\d \xi,
\end{equation*}
where the infimum is taken over all unit speed minimal geodesics $\gamma:[0,d(q,x)]\to M$ from $q$ to $x$.
For $r>0$,
we define
\begin{equation*}
B_{\psi,q}(r):=\left\{\,x\in M \mid s_q(x) < r \,\right\},
\end{equation*}
and also define measures
\begin{equation*}
\mu:=\e^{-\psi}\,v_g,\quad \nu:=\e^{ - \frac{2(1-\eps)\psi}{n-1} }\mu.
\end{equation*}
We set
\begin{equation*}
\mathcal{S}_{-K}(r):=\int^{r}_{0}\,\bs^{1/c}_{-K}(s)\,\d s
\end{equation*}
for $K > 0$.
In \cite{kuwae_sakurai}, they obtained the following theorem.
\begin{theorem}(\cite[Proposition 4.6]{kuwae_sakurai},Volume comparison)\label{deformed-bishop-gromov}
   Let $(M,g,\mu)$ be an $n$-dimensional weighted Riemannian manifold.
We assume $\Ric^{m}_{\psi} \geq -K\e^{ \frac{4(\eps-1)\psi}{n-1} }g$ for $K > 0$.
Then for all $r,R>0$ with $r\leq R$
we have
\begin{equation*}
\frac{\nu(B_{\psi,q}(R))}{\nu(B_{\psi,q}(r))} \leq \frac{\mathcal{S}_{-cK}(R)}{\mathcal{S}_{-cK}(r)}.
\end{equation*}
\end{theorem}
 
%%%%%%%%%%%%%%%%%%%%%
In the following argument, we start from Theorem \ref{deformed-bishop-gromov} instead of Theorem \ref{bishop_gromov} to prove a Cheng type inequality of the $L^p$-spectrum for the deformed measure $\nu$.
\begin{theorem}\label{appendix_thm_1}
   Let $(M,g,\mu)$ be a complete weighted Riemannian manifold.
   We assume that $s_q$ is smooth and there exists a constant $k > 0$ such that
   \begin{equation*}
       |\nabla s_q(x)|\leq k
   \end{equation*}
   holds for arbitrary $x\in M$. We also assume
   \begin{equation*}
       \Ric^{m}_{\psi} \geq -K\,\e^{ \frac{4(\eps-1)\psi}{n-1} }g
       \end{equation*}
       for $K > 0$. Then we have
       \begin{equation}\label{cheng-type-inequality-nu-1}
       \lambda_{\nu,p}(M) \leq \left(\frac{k}{p}\sqrt{\frac{K}{c}}\right)^p.
   \end{equation}
\end{theorem}
\begin{proof}
   We apply the argument in Theorem \ref{cheng-type-epsilon-range}.
 
For $R\geq 2$,
let $\eta:\mathbb{R}\rightarrow \mathbb{R}$ be a nonnegative smooth function such that $\eta = 1$ on $(-(R-1),R-1)$, $\eta = 0$ on $\mathbb{R}\backslash (-R,R)$ and $|\eta '| \leq C_3$, where $C_3$ is a constant independent of $R$.
We set, for an arbitrary $\delta > 0$,
\begin{equation*}
\alpha = -\frac{\sqrt{K/c} + \delta}{p}
\end{equation*}
and
\begin{equation*}
   \phi(y) := \exp(\alpha s_q(y))\varphi(y),
\end{equation*}
where $\varphi(y) := \eta(s_q(y))$.
By the assumption of $s_{q}$, we have
\begin{equation*}
   |\nabla \varphi| = |\eta'(s_q)||\nabla s_q| \leq k C_3.
\end{equation*}
As in the proof of Theorem \ref{cheng-type-epsilon-range}, we find for an arbitrary $\zeta > 0$,
\begin{eqnarray*}
   |\nabla \phi|^p &=&\left|\alpha \e^{\alpha s_q} \varphi \nabla s_q+\e^{\alpha s_q} \nabla \varphi\right|^p \\
   & \leq& \e^{p \alpha s_q}(-k\alpha \varphi+|\nabla \varphi|)^p \\
   & \leq& \e^{p \alpha s_q}\left[(1+\zeta)^{p-1}(-k\alpha \varphi)^p+\left(\frac{1+\zeta}{\zeta}\right)^{p-1}|\nabla \varphi|^p\right] .
   \end{eqnarray*}
By the definition of $\lambda_{\nu,p}(M)$, we obtain
\begin{eqnarray}
\lambda_{\nu,p}(M) & \leq&(1+\zeta)^{p-1}(-k\alpha)^p+\left(\frac{1+\zeta}{\zeta}\right)^{p-1} \frac{\int_M \exp(p \alpha s_q)|\nabla \varphi|^p \d\nu}{\int_M \exp(p \alpha s_q) \varphi^p \d\nu}\nonumber \\
&=&(1+\zeta)^{p-1}(-k\alpha)^p+\left(\frac{1+\zeta}{\zeta}\right)^{p-1} \frac{\int_{B_{\psi,q}(R) \backslash B_{\psi,q}(R-1)} \exp(p \alpha s_q)|\nabla \varphi|^p \d\nu}{\int_{B_{\psi,q}(R)} \exp(p \alpha s_q) \varphi^p \d\nu}\nonumber \\
& \leq&(1+\zeta)^{p-1}(-k\alpha)^p+(kC_3)^p\left(\frac{1+\zeta}{\zeta}\right)^{p-1} \frac{\exp(p \alpha (R-1)) \nu\left(B_{\psi,q}(R)\right)}{\int_{B_{\psi,q}(1)} \exp(p \alpha s_q) \d\nu} \nonumber\\
& \leq&(1+\zeta)^{p-1}(-k\alpha)^p+(kC_3)^p\left(\frac{1+\zeta}{\zeta}\right)^{p-1} \frac{\exp(p \alpha (R-1)) \nu\left(B_{\psi,q}(R)\right)}{\exp(p \alpha) \nu\left(B_{\psi,q}(1)\right)}\label{variational-principle-nu-1} .
\end{eqnarray}
From Theorem \ref{deformed-bishop-gromov} and
\begin{eqnarray*}
(\sqrt{cK})^{1/c}\mathcal{S}_{-cK}(R) &\leq& \int_0^R \left[\frac{1}{2}\left\{\exp(\sqrt{cK}s) - \exp(-\sqrt{cK}s)\right\}\right]^{1/c}\d s\\
&\leq& \sqrt{\frac{c}{K}}\exp\left(\sqrt{\frac{K}{c}}R\right),
\end{eqnarray*}
we deduce
\begin{eqnarray*}
   \frac{\e^{p\alpha (R-1)}\nu(B_{\psi,q}(R))}{\e^{p\alpha}\nu(B_{\psi,q}(1))}&\leq & \frac{\e^{p\alpha R}}{\e^{2p\alpha}}\frac{1}{(\sqrt{cK})^{1/c}\mathcal{S}_{-cK}(1)}\sqrt{\frac{c}{K}}\exp\left(\sqrt{\frac{K}{c}}R\right)\\
   &=& \frac{1}{\e^{2p\alpha}(\sqrt{cK})^{1/c} \mathcal{S}_{-cK}(1)}\sqrt{\frac{c}{K}}\exp\left(p\alpha R + \sqrt{\frac{K}{c}}R\right) \rightarrow 0
   \end{eqnarray*}
as $R\rightarrow \infty$. Letting $R\rightarrow \infty$ in \eqref{variational-principle-nu-1}, we obtain
\begin{equation}
\lambda_{\nu,p}(M) \leq (1 + \zeta)^{(p-1)}(-k\alpha)^p .
\end{equation}
Since $\zeta > 0$ and $\delta > 0$ are arbitrary, the theorem follows.
\end{proof}
%%%%%%%%%%%%%%%%%%%%%%%
\textbf{Acknowledgement.}

I would like to express deep appreciation to my supervisor Shin-ichi Ohta for his support, encouragement and for making a number of valuable suggestions and comments on preliminary versions of this paper.

%%%%%%%%%%%%%%%%%%%%%%%%

%%%%%%%%%%%%%%%%%%%%%

\end{document}